\documentclass[12pt,fleqn]{article}
\usepackage{amsfonts}
\usepackage{mathrsfs}
\usepackage[dvips]{graphics}
\usepackage[dvips]{color}
\usepackage{amsthm}
\usepackage[]{amsmath}
\usepackage{CJK}
\newtheorem{proposition}{Proposition}[section]
\newtheorem{lemma}[proposition]{Lemma}
\newtheorem{definition}[proposition]{Definition}
\newtheorem{Theorem}[proposition]{Theorem}
\newtheorem{corollary}[proposition]{Corollary}
\newtheorem{property}[proposition]{Property}
\begin{document}
\begin{CJK*}{GBK}{song}

\centerline{\Large{\textbf{Gap Sequence of Cutting Sequence with Slope $\theta=[0;\dot{d}]$}}}

\vspace{0.1cm}

\centerline{Yuke Huang\footnote[1]{Department of Mathematical Sciences, Tsinghua University, Beijing, 100084, P. R. China.}$^,$\footnote[2]{E-mail address: hyg03ster@163.com.}
~~Hanxiong Zhang\footnote[3]{College of Sciences, China University of Mining and Technology, Beijing, Beijing, 100083, P. R. China.}$^,$\footnote[4]{E-mail address: zhanghanxiong@163.com(Corresponding author).}}

\vspace{1cm}

\centerline{\textbf{\large{ABSTRACT}}}

\vspace{0.4cm}

In this paper, we consider the factor properties and gap sequence of a special type of cutting sequence with slope $\theta=[0;\dot{d}]$, denoted by $F_{d,\infty}$.
Let $\omega$ be a factor of $F_{d,\infty}$, then it occurs in the sequence infinitely many times. Let $\omega_p$ be the $p$-th occurrence of $\omega$ and $G_p(\omega)$ be the gap between $\omega_p$ and $\omega_{p+1}$.
We define the $d$ types of kernel words and envelope words,
give two versions of "uniqueness of kernel decomposition
property". Using them, we prove the gap sequence $\{G_p(\omega)\}_{p\geq1}$
has exactly two distinct elements for each $\omega$, and determine the expressions of gaps completely.
Furthermore, we prove that the gap sequence is $\sigma_i(F_{d,\infty})$, where $\sigma_i$ is a substitution depending only on the type of $Ker(\omega)$, i.e. the kernel word of $\omega$.
We also determine the position of $\omega_p$ for all $(\omega,p)$.
As applications, we study some combinatorial properties, such as the power, overlap and separate property between $\omega_p$ and $\omega_{p+1}$ for all $(\omega,p)$, and find all palindromes in $F_{d,\infty}$.

\vspace{0.2cm}



\vspace{0.6cm}

\setcounter{section}{1}

\noindent\textbf{\large{1.~Introduction}}

\vspace{0.4cm}

Let $\theta>0$ be an irrational real number, and consider the line $L_\theta: y=\theta x$ for $x>0$ through the origin with slope $\theta$. As the line $L_\theta$ travels to the right, write $c_i=a$ if $L_\theta$ intersects a vertical line; $c_i=b$ if $L_\theta$ intersects a horizontal line.
Call the resulting infinite word $C_\theta=c_1c_2c_3\cdots$ a cutting sequence with slope $\theta$. For example, for $\theta=(\sqrt{5}-1)/2$, $C_\theta$ is the Fibonacci sequence $abaababa\cdots$. Cutting sequence, as a kind of aperiodic sequence with minimal language complexity, have been studied for a long time.
These sequences appear in the mathematical literature under many different names, such as Sturmian sequences, rotation sequences, Christoffel words, balanced sequences, and so forth.

The combinatorial properties of cutting sequence are of great
interest in many aspects of mathematics and computer science, symbolic dynamics, theoretical computer science etc., we refer to Allouche and Shallit\cite{AS2003}, Lothaire\cite{L1983,L2002}, Berstel\cite{B1966,B1980}.
Cutting sequence have also been considered by
Wen and Wen\cite{WW1993}, Ito and Yasutomi\cite{IY1990}, Mignosi\cite{M1991},
Cao and Wen\cite{CW2003}, Chuan and Ho\cite{CH2010}, and so forth.

Let slope $\theta$ have a continued fraction expansion $\theta=[0;d_1,d_2,d_3,\cdots]$ with $d_i\in\mathbb{N}$.
In this paper, we consider a special type of cutting sequences with slope $\theta=[0;\dot{d}]$, i.e. $d_i=d$ for all $i\geq1$. In this case, we denote the cutting sequence by $F_{d,\infty}$. Since $(\sqrt{5}-1)/2=[0;1,1,1,\cdots]$, $F_{1,\infty}$ is Fibonacci sequence.

Wen and Wen\cite{WW1994} studied the factor structure of Fibonacci sequence, where they defined the singular word and give the positively separate property of the singular words.
Huang and Wen\cite{HW2014} extend the results from singular words to arbitrary words $\omega$ of the Fibonacci sequence, and discuss the structure of gap sequence $\{G_p(\omega)\}_{p\geq1}$.
The main aim of this article is to extend the results in Huang and Wen\cite{HW2014} from Fibonacci sequence $F_{1,\infty}$ to sequences $F_{d,\infty}$ for $d\geq2$.

The main result in this paper is as follows.

\vspace{0.2cm}

\noindent\textbf{Theorem} (Gap sequence of factor $\omega\prec F_{d,\infty}$)\textbf{.}

\emph{(1) Any factor $\omega$ has exactly two distinct gaps;}

\emph{(2) The gap sequence $\{G_p(\omega)\}$ is
the sequence $\sigma_i(F_{d,\infty})$, where $\sigma_i$ is a substitution depending only on the type of $Ker(\omega)$, i.e. the kernel of $\omega$.}

\vspace{0.2cm}

The main tools in this paper are "kernel word" and "envelope word". Using them, we can give the expressions of each gap $G_p(\omega)$ and each substitution $\sigma_i$,
Then we can determine the structure of gap sequence of $F_{d,\infty}$ completely for all $d\geq2$. We can also give the position of $\omega_p$.

This paper is organized as follows.

Section 1 is devoted to the introduction and preliminaries. In Section 2, we define two new notions "kernel word" and "envelope word". In Section 3 and 4, we discuss the gaps and gap sequence of kernel words and envelope words separately, then give the relation between them in Section 5. In Section 6, we give two versions of "uniqueness of kernel decomposition property": weak and strong, these properties make "kernel word" so important and so special. Using them, we can determine the gaps and gap sequence of arbitrary word.
As applications, we study some combinatorial properties in Section 7, such as the power, overlap and separate property between $\omega_p$ and $\omega_{p+1}$ for all $(\omega,p)$, and find all palindromes with kernel $K_{d,m,i}$.

\vspace{0.4cm}

\noindent\emph{1.1 Notation and Basic Properties}

\vspace{0.4cm}

Let $\mathcal{A}=\{a,b\}$ be a binary alphabet. Let $\mathcal{A}^\ast$ be the set of finite words on $\mathcal{A}$ and $\mathcal{A}^{\mathbb{N}}$ be the set of one-sided infinite words. The elements of $\mathcal{A}^\ast$ are called words or factors, which will be denoted by $\omega$. The neutral element of $\mathcal{A}^\ast$ is called the empty word, which we denote by $\varepsilon$. For a finite word $\omega=x_1x_2\cdots x_n$, the length of $\omega$ is equal to $n$ and denoted by $|\omega|$.
The number of occurrences of letter $\alpha\in\mathcal{A}$ in $\omega$ is denoted by $|\omega|_\alpha$.

Let $F_{d,\infty}$ be the cutting sequence with slope $\theta=[0;\dot{d}]$. Let $\sigma_d:\mathcal{A}\rightarrow \mathcal{A}^\ast$ be a morphism defined by $\sigma_d(a)=a^db$, $\sigma_d(b)=a$. As we know, $\mathcal{A}^\ast$ is the free monoid on $\mathcal{A}$, so $\sigma_d(ab)=\sigma_d(a)\sigma_d(b)$. We define the $m$-th iteration of $\sigma_d$ by $\sigma_d^m(a)=\sigma_d^{m-1}(\sigma_d(a))$, $m\geq1$ and we denote $F_{d,m}=\sigma_d^m(a)$. By convention, we define $\sigma_d^0(a)=a$ and $\sigma_d^0(b)=b$.
Then the fixed point beginning with $a$ of the substitution $\sigma_d$ is sequence $F_{d,\infty}$. For details, see Theorem 3 in \cite{CW2003}.

The notation $\nu\prec\omega$ means that word $\nu$ is a factor of word $\omega$.
We say that word $\nu$ is a prefix (resp. suffix) of word $\omega$, and write $\nu\triangleleft\omega$ (resp. $\nu\triangleright\omega$) if there exists $u\in\mathcal{A}^\ast$ such that $\omega=\nu u$ (resp. $\omega=u\nu$).

For a finite word $\omega=x_1x_2\cdots x_n$, the mirror word $\overleftarrow{\omega}$ of $\omega$ is defined to be $\overleftarrow{\omega}=x_n\cdots x_2x_1$. A word $\omega$ is called a palindrome if $\omega=\overleftarrow{\omega}$.

Let $\tau=x_1x_2\cdots$ be a sequence, for any $i\leq j$, define $\tau[i,j]:=x_ix_{i+1}\cdots x_{j-1}x_j$, the factor of $\tau$ of length $j-i+1$, starting from the $i$-th letter and ending to the $j$-th letter. By convention, we note $\tau[i]:=\tau[i,i]=x_i$ and $\tau[i,i-1]:=\varepsilon$.

If $\nu\prec\omega$, where $\omega=x_1x_2\cdots$ is a finite word or a sequence, $\nu$ is said to occur at position $i$ in $\omega$ if $\omega[i,i+|\nu|-1]=x_ix_{i+1}\cdots x_{i+|\nu|-1}=\nu$.

Let $\nu=\nu_1\nu_2\cdots\nu_n\in\mathcal{A}^\ast$, we denote by $\nu^{-1}:=\nu_n^{-1}\cdots\nu_2^{-1}\nu_1^{-1}$, called the inverse word of $\nu$.
Let $\omega=u\nu$, then $\omega^{-1}=(u\nu)^{-1}=\nu^{-1}u^{-1}$, $\omega\nu^{-1}=u\nu\nu^{-1}=u$ and $u^{-1}\omega=u^{-1}u\nu=\nu$.

\vspace{0.4cm}

\noindent\emph{1.2~Some Definitions}

\vspace{0.4cm}

Let $\omega$ be factor of cutting sequence $F_{d,\infty}$ for $d\geq2$. In this subsection, we will introduce some definitions: factor sequence $\{\omega_p\}_{p\ge 1}$, gap word $G_p(\omega)$, gap sequence $\{G_p(\omega)\}_{p\ge 1}$, etc.
We will give the definitions about kernel word and envelope word in Section 2.

\begin{definition}[Factor sequence]
Let $\omega$ be a factor of cutting sequence $F_{d,\infty}$ for $d\geq2$, then it occurs in the sequence infinitely many times, which we arrange by the sequence
$\{\omega_p\}_{p\ge 1}$, where $\omega_p$ denote the $p$-th occurrence of $\omega$.
\end{definition}

\begin{definition}[Gap]
Let $\omega_p=x_{i+1}\cdots x_{i+n}$, $\omega_{p+1}=x_{j+1}\cdots x_{j+n}$, the gap
between $\omega_p$ and $\omega_{p+1}$, denoted by $G_p(\omega)$, is defined by
\begin{equation*}
G_p(\omega)=
\begin{cases}
\varepsilon&when~i+n=j,~\omega_p~and~\omega_{p+1}~are~adjacent;\\
x_{i+n+1}\cdots x_{j}&when~i+n<j,~\omega_p~and~\omega_{p+1}~are~separated;\\
(x_{j+1}\cdots x_{i+n})^{-1}&when~i+n>j,~\omega_p~and~\omega_{p+1}~are~overlapped.
\end{cases}
\end{equation*}

The set of gaps of factor $\omega$ is defined by $\{G_p(\omega)|~p\geq1\}$.
\end{definition}

\noindent\textbf{Example.} In sequence $F_{2,\infty}$,
consider factor $a$, $G_1(a)=\varepsilon$ (adjacent) and $G_2(a)=b$ (separated);
consider factor $aa$, $G_1(aa)=b$ (separated) and $G_3(aa)=a^{-1}$ (overlapped).

\vspace{0.2cm}

\noindent\textbf{Remark.}
1. By convention, we define $G_0(\omega)$ as the prefix of $F_{d,\infty}$ before $\omega_1$.

2. When $\omega_p$ and $\omega_{p+1}$ are overlapped,
the overlapped part is the word $x_{j+1}\cdots x_{i+n}$. We take its inverse word as the gap $G_p(\omega)$.
By this way, it is clear to distinguish the cases "adjacent", "separated" and "overlapped".

\vspace{0.2cm}
%
%

\noindent\textbf{Remark.}
A related concept of "gap" is "return word", which is introduced by F.Durand\cite{D1998}.
He proved that a sequence is primitive substitutive if and only if the set of its return words is finite, which means, each factor of this sequence has finite return words.
In 2001, L.Vuillon\cite{V2001} proved that an infinite word $\tau$ is a Sturmian sequence if and only if each non-empty factor $\omega\prec\tau$ has exactly two distinct return words.
Some other related researches(see also \cite{AB2005,BPS2008}) were interested in the cardinality of the set of return words of $\omega$ and the consequent results, but didn't concern about the structures of the sequence derived by return words.

Essentially, gap words can be derived from the return words which differ from only one prefix $\omega$, but since the terminology "gap" will be convenient and have some advantages for our discussions, we prefer to adopt it.

\begin{definition}[Gap sequence]
Let $G_p(\omega)$ be the gap between $\omega_p$ and $\omega_{p+1}$, we call $\{G_p(\omega)\}_{p\ge 1}$ the gap sequence of factor $\omega$.
\end{definition}



\vspace{0.5cm}

\stepcounter{section}

\noindent\textbf{\large{2.~Kernel Words $K_{d,m,i}$ and Envelope Word $E_{d,m,i}$}}

\vspace{0.4cm}

In Huang and Wen\cite{HW2014}, we defined a new concept "kernel word", which plays an important role in the research. We are going to determine the kernel words and envelope word for sequence $F_{d,\infty}$, $d\geq2$. Using them, we can study the structure of sequence $F_{d,\infty}$.

Notice that in the case of Fibonacci sequence, the kernel words are exactly
singular words, but this is not the case for sequence $F_{d,\infty}$. In fact, there are $d$ types of kernel words for sequence $F_{d,\infty}$, and they are much more complicated.

\begin{definition}[$\delta_m$] Let $\delta_m$ be the last letter of $F_{d,m}=\sigma_d^m(a)$, where $\sigma_d(a,b)=(a^db,a)$. \end{definition}

\begin{property} For all $d\geq2$, $\delta_m=a$ (resp. $b$) when $m$ is even (resp. odd). \end{property}

\begin{definition}[Kernel word $K_{d,m,i}$] The cutting sequence $F_{d,\infty}$ has $d$ types of kernel words. The kernel word with order $m$ of $i$-th type is defined as
$$K_{d,m,i}=\delta_m\ast F_{d,m}^i\ast F_{d,m-1}\ast \delta_{m-1}^{-1},$$
where $\delta_m$ is the last letter of $F_{d,m}$, $d\geq1$, $m\geq0$ and $0\leq i\leq d-1$.
\end{definition}

\noindent\textbf{Remark.} When $d=1$, the cutting sequence $F_{d,\infty}$ is Fibonacci sequence. There is only one type of kernel words: singular words.
When $d\geq2$, there are $d$ types of kernel words, two of them are singular word and adjoining word (see Cao and Wen\cite{CW2003}).

\vspace{0.2cm}

\noindent\textbf{Example.} $K_{3,1,0}=b$ (singular word), $K_{3,1,1}=baaab$,  $K_{3,1,2}=baaabaaab$ (adjoining word).

\begin{definition}[Kernel set] For fixed $d\geq2$, kernel set $\mathcal{K}_d:=\{K_{d,m,i}:m\geq0,0\leq i\leq d-1\}$. \end{definition}

\begin{definition}[Order] For fixed $d\geq2$, we give an order "$\sqsubset$" on kernel set $\mathcal{K}_d$ as follows:

(1) If $m<n$, then $K_{d,m,i}\sqsubset K_{d,n,j}$;

(2) If $m=n$ and $i<j$, then $K_{d,m,i}\sqsubset K_{d,n,j}$.
\end{definition}

\begin{definition}[Kernel word of factor $\omega$, $Ker(\omega)$]\
Let $\omega$ be a factor of sequence $F_{d,\infty}$, the kernel word of factor $\omega$ denote by $Ker(\omega)=K_{d,m,i}\in\mathcal{K}$, where
$$K_{d,m,i}=\max_\sqsubset\{K_{d,n,j}:K_{d,n,j}\prec\omega\}$$
\end{definition}

\begin{lemma}\

(1) $K_{d,m,0}=K_{d,m-2,d-1}\ast K_{d,m-2,d-2}^{-1}\ast K_{d,m-2,d-1}$ when $m\geq2$;

(2) $K_{d,m,i}=[K_{d,m,0}\ast K_{d,m-1,d-1}]^i\ast K_{d,m,0}$ when $m\geq1$ and $1\leq i\leq d-1$.
\end{lemma}

\begin{proof} The proofs of (1) and (2) are similar, we prefer to take (2) for example.

When $m\geq1$ and $1\leq i\leq d-1$, by the definition of kernel word and $\delta_m=\delta_{m-2}$, we have
\begin{equation*}
\begin{split}
&[K_{d,m,0}\ast K_{d,m-1,d-1}]^i\ast K_{d,m,0}\\
=&[\delta_mF_{d,m-1}\delta_{m-1}^{-1}\ast \delta_{m-1}F_{d,m-1}^{d-1}F_{d,m-2}\delta_{m-2}^{-1}]^i\ast \delta_mF_{d,m-1}\delta_{m-1}\\
=&[\delta_mF_{d,m-1}^{d}F_{d,m-2}\delta_{m-2}^{-1}]^i\ast \delta_mF_{d,m-1}\delta_{m-1}
=\delta_mF_{d,m}^iF_{d,m-1}\delta_{m-1}=K_{d,m,i}
\end{split}
\end{equation*}

The proof of (1) could be obtained by a similar argument.
\end{proof}

Using the property above and by induction, we have the next corollary.

\begin{property}[Palindrome] Each kernel word $K_{d,m,i}$ is palindrome.\end{property}

\begin{definition}[Envelope words $E_{d,m,i}$] The cutting sequence $F_{d,\infty}$ has $d$ types of envelope words. The envelope word with order $m$ of $i$-th type is defined as
$$E_{d,m,i}=F_{d,m}^{i+1}\ast F_{d,m-1}\ast F_{d,m}\ast\delta_{m}^{-1}\delta_{m-1}^{-1},$$
where $\delta_m$ is the last letter of $F_{d,m}$, $d\geq1$, $m\geq0$ and $0\leq i\leq d-1$.
\end{definition}

\begin{proposition} The relation between $K_{d,m,i}$ and $E_{d,m,i}$ is
$$E_{d,m,i}=\mu_1\ast K_{d,m,i}\ast\mu_2,$$
where $\mu_1=\delta_{m+1}^{-1}K_{d,m+1,0}$ and $\mu_2=K_{d,m+1,0}\delta_{m+1}^{-1}$ are constant words depending only on $d$, $m$.
\end{proposition}

Since each kernel word $K_{d,m,i}$ is palindrome, using the relation between $K_{d,m,i}$ and $E_{d,m,i}$ above, we have the next corollary.

\begin{property}[Palindrome] Each envelope word $E_{d,m,i}$ is palindrome.\end{property}

In Section 6, we give two versions of "uniqueness of kernel decomposition property": weak and strong, these properties make "kernel word" so important and so special.
In fact, we can prove the kernel set $\mathcal{K}_d$ we defined is minimum.

\begin{property}[$\mathcal{K}_d$ is minimum]\ Each subset of kernel set $\mathcal{K}_d$ can not satisfy the "Uniqueness of kernel decomposition, weak" in Theorem 6.2.
\end{property}

\begin{proof} Let $\mathcal{K}'$ is a proper subset of $\mathcal{K}_d$, $K_{d,m,i}$ is the minimal element of $\mathcal{K}_d-\mathcal{K}'$ under order $"\sqsubset"$.
If $K_{d,m,i}=a$ (resp. $b$), then factor $\omega=a$ (resp. $b$) has no kernel word in $\mathcal{K}'$.

(1) If $i=0$, by Lemma 2.7 and induction, the kernel word in $\mathcal{K}'$ of $\omega=K_{d,m,0}$ is $K_{d,m-2,d-1}$.

(2) If  $1\leq i\leq d-1$, by Lemma 2.7, the kernel word in $\mathcal{K}'$ of $\omega=K_{d,m,i}$ is $K_{d,m,i-1}$.

Both (1) and (2), $Ker(\omega)$ occurs in $\omega$ twice.
So $\mathcal{K}'$ can not satisfy the "Uniqueness of kernel decomposition, weak" in Theorem 6.2.
\end{proof}


\vspace{0.5cm}

\stepcounter{section}

\noindent\textbf{\large{3.~Gaps and Gap Sequence of Kernel Word $K_{d,m,i}$}}

\vspace{0.4cm}
%
%

In this section, we will determine the structure of the gap sequence $\{G_p(K_{d,m,i})\}_{p\geq1}$ for each kernel word $K_{d,m,i}$.
First, we give two lemmas about the basic properties of $F_{d,m}$.

\begin{lemma}\

(1) $F_{d,m}$ occurs in $F_{d,m}F_{d,m}$ twice at positions 1 and $f_{d,m}+1$ when $m\geq1$;

(2) $F_{d,m}$ occurs in $F_{d,m}F_{d,m-1}F_{d,m}$ three times at positions 1, $f_{d,m}+1$ and $f_{d,m}+f_{d,m-1}+1$ when $m\geq2$.
\end{lemma}

\begin{proof} By induction, when $m=1,2$ (resp. $m=2$), property (1) (resp. (2)) holds. Assume the two properties hold for $n=m$. Consider $n=m+1$, i.e. the position of $F_{d,m+1}$.

\vspace{0.2cm}

\textbf{Proof of (1).} Consider the first $F_{d,m}$ in $F_{d,m+1}=\underline{F_{d,m}}\cdots F_{d,m}F_{d,m-1}$, the $d+2$ possible positions are shown as $[i]$ in the next figure, which is plotted as $d=3$.
\setlength{\unitlength}{0.9mm}
\begin{center}
\begin{picture}(135,28)
\linethickness{3pt}
\put(0,6){\line(1,0){132}}
\linethickness{2pt}
\put(20,20){\line(1,0){66}}
\linethickness{1pt}
\put(0,0){\line(0,1){18}}
\put(66,0){\line(0,1){18}}
\put(132,0){\line(0,1){18}}
\put(20,0){\line(0,1){10}}
\put(40,0){\line(0,1){10}}
\put(60,0){\line(0,1){10}}
\put(86,0){\line(0,1){10}}
\put(106,0){\line(0,1){10}}
\put(126,0){\line(0,1){10}}
\put(6,8){$F_{d,m}$}
\put(26,8){$F_{d,m}$}
\put(46,8){$F_{d,m}$}
\put(72,8){$F_{d,m}$}
\put(92,8){$F_{d,m}$}
\put(112,8){$F_{d,m}$}
\put(24,14){\vector(-1,0){24}}
\put(42,14){\vector(1,0){24}}
\put(26,13){$F_{d,m+1}$}
\put(90,14){\vector(-1,0){24}}
\put(108,14){\vector(1,0){24}}
\put(92,13){$F_{d,m+1}$}
\put(0,1){$[1]$}
\put(20,1){$[2]$}
\put(40,1){$[3]$}
\put(60,1){$[4]$}
\put(66,1){$[5]$}
\put(47,22){$F_{d,m+1}$}
\put(12,22){If:}
\put(20,20){\line(0,1){6}}
\put(86,20){\line(0,1){6}}
\put(44,23){\vector(-1,0){24}}
\put(62,23){\vector(1,0){24}}
\end{picture}
\end{center}
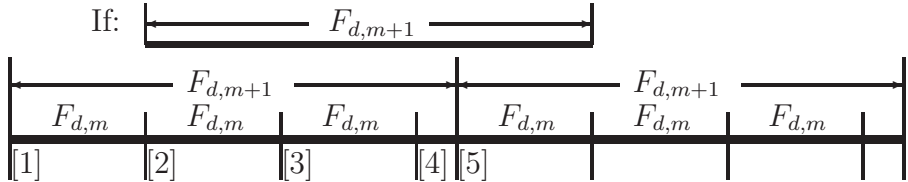
\centerline{Fig. 3.1: The $d+2$ possible positions in (1).}

\vspace{0.2cm}

Obviously, $F_{d,m+1}$ occurs at positions [1] and [5]. We are going to exclude all other possible positions. Take position [2] for example.
Suppose $F_{d,m+1}$ occurs at position [2], the last letter is $\delta_{m+1}$. But at the same position in $F_{d,m+1}F_{d,m+1}$, the letter is $\delta_{m}$. It contradicts $\delta_{m+1}\neq\delta_{m}$.

\vspace{0.2cm}

\textbf{Proof of (2).} Consider the first $F_{d,m}$ in $F_{d,m+1}=\underline{F_{d,m}}\cdots F_{d,m}F_{d,m-1}$, the $2d+3$ possible positions are shown as $[i]$ in the next figure, which is plotted as $d=3$.
\setlength{\unitlength}{0.9mm}
\begin{center}
\begin{picture}(155,28)
\linethickness{3pt}
\put(0,6){\line(1,0){152}}
\linethickness{2pt}
\put(20,20){\line(1,0){66}}
\linethickness{1pt}
\put(0,0){\line(0,1){18}}
\put(66,0){\line(0,1){18}}
\put(86,0){\line(0,1){18}}
\put(152,0){\line(0,1){18}}
\put(20,0){\line(0,1){10}}
\put(40,0){\line(0,1){10}}
\put(60,0){\line(0,1){10}}
\put(72,0){\line(0,1){10}}
\put(78,0){\line(0,1){10}}
\put(84,0){\line(0,1){8}}
\put(106,0){\line(0,1){10}}
\put(126,0){\line(0,1){10}}
\put(146,0){\line(0,1){10}}
\put(6,8){$F_{d,m}$}
\put(26,8){$F_{d,m}$}
\put(46,8){$F_{d,m}$}
\put(92,8){$F_{d,m}$}
\put(112,8){$F_{d,m}$}
\put(132,8){$F_{d,m}$}
\put(24,14){\vector(-1,0){24}}
\put(42,14){\vector(1,0){24}}
\put(26,13){$F_{d,m+1}$}
\put(110,14){\vector(-1,0){24}}
\put(128,14){\vector(1,0){24}}
\put(112,13){$F_{d,m+1}$}
\put(72,13){$F_{d,m}$}
\put(0,1){$[1]$}
\put(20,1){$[2]$}
\put(40,1){$[3]$}
\put(60,1){$[4]$}
\put(66,1){$[5]$}
\put(72,1){$[6]$}
\put(78,1){$[7]$}
\put(81,9){$[8]$}
\put(86,1){$[9]$}
\put(14,22){If:}
\put(46,22){$F_{d,m+1}$}
\put(20,20){\line(0,1){6}}
\put(86,20){\line(0,1){6}}
\put(44,23){\vector(-1,0){24}}
\put(62,23){\vector(1,0){24}}
\end{picture}
\end{center}
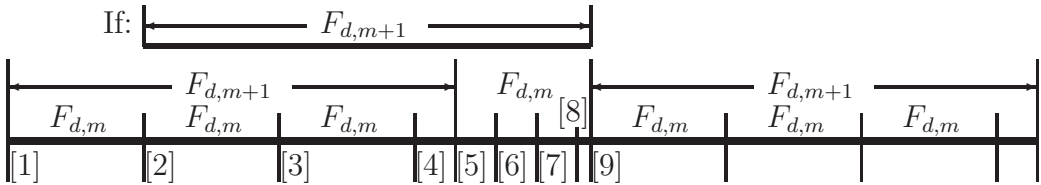
\centerline{Fig. 3.2: The $2d+3$ possible positions in (2).}

\vspace{0.2cm}

Obviously, $F_{d,m+1}$ occurs at positions [1] and [9]. Since
$$F_{d,m}F_{d,m-1}F_{d,m}=F_{d,m}F_{d,m-1}^{d+1}F_{d,m-2}=F_{d,m} \underline{F_{d,m-1}^dF_{d,m-2}}F_{d,m-2}^{d-1}F_{d,m-3}F_{d,m-2},$$
then $F_{d,m+1}$ occurs at position [5]. We are going to exclude all other possible positions.

(1) Suppose $F_{d,m+1}$ occurs at positions [2] or [3], the last letter is $\delta_{m+1}$. But at the same position in $F_{d,m+1}F_{d,m}F_{d,m+1}$, the letter is $\delta_{m}$. It contradicts $\delta_{m+1}\neq\delta_{m}$.

(2) Suppose $F_{d,m+1}$ occurs at position [4], the second $F_{d,m}$ in
$F_{d,m+1}=F_{d,m}\underline{F_{d,m}}F_{d,m}F_{d,m-1}$ occurs in $F_{d,m}F_{d,m}$ as below with position unequal to $1$ or $f_{d,m}$.
$$F_{d,m+1}F_{d,m}F_{d,m+1}=F_{d,m}F_{d,m}F_{d,m}F_{d,m-1}\ast \underline{F_{d,m}\ast F_{d,m}}F_{d,m}F_{d,m}F_{d,m-1}.$$
It contradicts our assumption. Similarly, we know $F_{d,m+1}$ cannot occur at positions [6]-[8].
\end{proof}

\noindent\textbf{Remark.} $F_{d,1}=a^db$ occurs in $F_{d,1}F_{d,1-1}F_{d,1}=a^dba^{d+1}b$ only twice at position 1 and $f_{d,1}+1$.

\begin{lemma} For $m\geq1$,

(1) $F_{d,m}\delta_m^{-1}$ occurs in $F_{d,m}F_{d,m}$ twice at positions 1 and $f_{d,m}+1$;

(2) $F_{d,m}\delta_m^{-1}$ occurs in $F_{d,m}F_{d,m-1}F_{d,m}$ three times at positions 1, $f_{d,m}+1$ and $f_{d,m}+f_{d,m-1}+1$.
\end{lemma}

\begin{proof} The proof could be obtained by Lemma 3.1 and by a similar argument.\end{proof}

Using the two lemmas above, we can determine the expression of $G_0(K_{d,m,i})$ as follow.

\begin{Theorem} The prefix of $F_{d,\infty}$ before $K_{d,m,i,1}$ is
$F_{d,\infty}[1,f_{d,m}-1]$, denoted by $G_0(K_{d,m,i})$.
\end{Theorem}

\begin{proof} When $m=0$, $K_{d,0,i}=a^{i+1}$, which occurs at position 1 in $F_{d,\infty}$. So $G_0(K_{d,0,i})=\varepsilon$, the theorem holds. When $m\geq1$, there are two steps.

\vspace{0.2cm}

\textbf{Step 1.} We are going to show $K_{d,m,i}$ occurs at position $f_{d,m}$.

(1) When $0\leq i\leq d-2$. Since $F_{d,m}\triangleleft F_{d,\infty}$ and $F_{d,m}=F_{d,m-1}^dF_{d,m-2}$, then
$$F_{d,m}^d=F_{d,m}\delta_{m}^{-1}\ast
\underline{\delta_{m}F_{d,m}^{i}F_{d,m-1}\delta_{m-1}^{-1}}
\ast\delta_{m-1}F_{d,m-1}^{d-1}F_{d,m-2}F_{d,m}^{d-i-2}.$$

(2) When $i=d-1$. Since $F_{d,m}F_{d,m-1}\triangleleft F_{d,\infty}$, then
$$F_{d,m}^dF_{d,m-1}=F_{d,m}\delta_{m}^{-1}\ast \underline{\delta_{m}F_{d,m}^{d-1}F_{d,m-1}\delta_{m-1}^{-1}}\ast \delta_{m-1}.$$

So $K_{d,m,i}=\delta_m\ast F_{d,m}^i\ast F_{d,m-1}\ast \delta_{m-1}^{-1}$ occurs in $F_{d,m}$ at position $f_{d,m}$ when $m\geq1$.

\vspace{0.2cm}

\textbf{Step 2.} We are going to show $f_{d,m}$ is the first position of $K_{d,m,i}$.

(1) When $i=0$, $K_{d,m,0}=\delta_{m}\underline{F_{d,m-1}\delta_{m-1}^{-1}}$. Using Lemma 3.2, the $d+2$ possible positions of $F_{d,m-1}\delta_{m-1}^{-1}$ in $F_{d,m}F_{d,m-1}\delta_{m-1}^{-1}$ are shown as $[i]$ in the next figure, which is plotted as $d=3$.
\setlength{\unitlength}{1mm}
\begin{center}
\begin{picture}(90,18)
\linethickness{3pt}
\put(0,6){\line(1,0){86}}
\linethickness{1pt}
\put(0,0){\line(0,1){18}}
\put(66,0){\line(0,1){18}}
\put(20,0){\line(0,1){10}}
\put(40,0){\line(0,1){10}}
\put(60,0){\line(0,1){10}}
\put(86,0){\line(0,1){10}}
\put(4,8){$F_{d,m-1}$}
\put(24,8){$F_{d,m-1}$}
\put(44,8){$F_{d,m-1}$}
\put(70,8){$F_{d,m-1}$}
\put(25,15){\vector(-1,0){25}}
\put(41,15){\vector(1,0){25}}
\put(29,14){$F_{d,m}$}
\put(0,1){$[1]$}
\put(20,1){$[2]$}
\put(40,1){$[3]$}
\put(60,1){$[4]$}
\put(66,1){$[5]$}
\end{picture}
\end{center}
\centerline{Fig. 3.3: The $d+2$ possible positions of $F_{d,m-1}$.}

\vspace{0.2cm}

Obviously, the $F_{d,m-1}\delta_{m-1}^{-1}$ in $K_{d,m,0}=\delta_{m}\underline{F_{d,m-1}\delta_{m-1}^{-1}}$ occurs at position [5], we are going to exclude all other possible positions. Since there is no letters ahead, $F_{d,m-1}\delta_{m-1}^{-1}$ in $K_{d,m,i}$ cannot occurs at position [1].
Since the letter ahead position [2] (resp. [3] and [4]) is $\delta_{m-1}$, which is not equal to $\delta_{m}$, so $F_{d,m-1}\delta_{m-1}^{-1}$ in $K_{d,m,i}$ cannot occurs at this position too.

(2) When $1\leq i\leq d-1$. Consider the first $F_{d,m}$ in $K_{d,m,i}=\delta_{m}\underline{F_{d,m}}F_{d,m}^{i-1}F_{d,m-1}\delta_{m-1}^{-1}$,
Suppose $f_{d,m}$ is not the first position of $K_{d,m,i}$, then by Lemma 3.1, the $F_{d,m}$ occurs at position 1. Since $K_{d,m,i}$ has a letter $\delta_{m}$ before the $F_{d,m}$, so $K_{d,m,i}$ can not occur before position $f_{d,m}$.
\end{proof}

\begin{definition}[$L(\omega,p)$] Let $L(\omega,p)$ be the position of the $p$-th occurrence of factor $\omega$.\end{definition}

\begin{Theorem}[First two distinct gaps of kernel word $K_{d,m,i}$]\

Let $G_A(K_{d,m,i})=G_1(K_{d,m,i})$ and $B=\min\{p:G_p(K_{d,m,i})\neq G_1(K_{d,m,i})\}$,
then
\begin{equation*}
(1)~G_A(K_{d,m,i})=\begin{cases}
\varepsilon~(if~m=0)~\text{or}~K_{d,m-1,d-1}~(if~m\geq1),&i=0;\\
K_{d,m,i-1}^{-1},&1\leq i\leq d-2;\\
K_{d,m+1,0},&i=d-1
\end{cases}
\end{equation*}
\begin{equation*}
(2)~G_B(K_{d,m,i})=\begin{cases}
K_{d,m+1,0},&0\leq i\leq d-2;\\
K_{d,m,d-2}^{-1},&i=d-1.
\end{cases}
\end{equation*}
\begin{equation*}
(3)~B=\begin{cases}
d-i,&0\leq i\leq d-2;\\
d+1,&i=d-1.
\end{cases}
\end{equation*}
\end{Theorem}

\begin{proof} The proofs of $i=0$, $1\leq i\leq d-2$ and $i=d-1$ are similar, we prefer to take $1\leq i\leq d-2$ for example. By Theorem 3.3, we know $L(K_{d,m,i},1)=f_{d,m}$.
Furthermore
$$F_{d,m+1}=F_{d,m}F_{d,m}\delta_m^{-1}\delta_mF_{d,m}^iF_{d,m}^{d-i-2}F_{d,m-1}.$$
Since $d-i-2\geq0$, $K_{d,m,i}$ is at position $2\ast f_{d,m}$. Using Lemma 3.1, $F_{d,m}$ occurs in $F_{d,m}F_{d,m}$ only twice, so the position of $K_{d,m,i}$ cannot be between $f_{d,m}$ to $2\ast f_{d,m}$, i.e $L(K_{d,m,i},2)=2\ast f_{d,m}$.

Similarly, we can determine the positions of the first $d-i+1$ kernel word $K_{d,m,i}$:

(a) $L(K_{d,m,i},p)=p\ast f_{d,m}$ for $p=1,2,\ldots,d-i$;

(b) $L(K_{d,m,i},d-i+1)=f_{d,m+1}+f_{d,m}$.

By the definition of $B$, we know $B=d-i$.

\vspace{0.2cm}

The next figure is plotted as $d=4$ and $i=2$. In this case, $B=d-i=2$.
The figure shows the positions of $K_{d,m,i,p}$ for $p=1,2,3$. Using this, we can determine the expressions of $G_A(K_{d,m,i})$ and $G_B(K_{d,m,i})$.

\footnotesize
\setlength{\unitlength}{0.8mm}
\begin{center}
\begin{picture}(175,27)
\linethickness{3pt}
\put(0,6){\line(1,0){172}}
\linethickness{2pt}
\put(18,14){\line(1,0){46}}
\put(38,22){\line(1,0){46}}
\put(104,17){\line(1,0){46}}
\linethickness{1pt}
\put(0,0){\line(0,1){10}}
\put(20,6){\line(0,1){4}}
\put(40,6){\line(0,1){4}}
\put(60,6){\line(0,1){4}}
\put(80,6){\line(0,1){4}}
\put(86,0){\line(0,1){10}}
\put(106,6){\line(0,1){4}}
\put(126,6){\line(0,1){4}}
\put(146,6){\line(0,1){4}}
\put(166,6){\line(0,1){4}}
\put(172,0){\line(0,1){10}}
\put(6,8){$F_{d,m}$}
\put(26,8){$F_{d,m}$}
\put(46,8){$F_{d,m}$}
\put(66,8){$F_{d,m}$}
\put(92,8){$F_{d,m}$}
\put(112,8){$F_{d,m}$}
\put(132,8){$F_{d,m}$}
\put(152,8){$F_{d,m}$}
\put(37,1){$F_{d,m+1}$}
\put(123,1){$F_{d,m+1}$}
\put(35,3){\vector(-1,0){35}}
\put(51,3){\vector(1,0){35}}
\put(121,3){\vector(-1,0){35}}
\put(137,3){\vector(1,0){35}}
\put(34,16){$K_{d,m,i,1}$}
\put(18,14){\line(0,1){2}}
\put(64,14){\line(0,1){2}}
\put(54,24){$K_{d,m,i,2}$}
\put(38,22){\line(0,1){2}}
\put(84,22){\line(0,1){2}}
\put(120,19){$K_{d,m,i,3}$}
\put(104,17){\line(0,1){2}}
\put(150,17){\line(0,1){2}}
\end{picture}
\end{center}
\normalsize
\centerline{Fig. 3.4: The expressions of gaps $G_A(K_{d,m,i})$ and $G_B(K_{d,m,i})$.}

\vspace{0.2cm}

By Fig. 3.4, we know the expressions of the gaps are:

(1') $G_A(K_{d,m,i})=[\delta_mF_{d,m}^{i-1}F_{d,m-1}\delta_{m-1}^{-1}]^{-1}=K_{d,m,i-1}^{-1}$;

(2') $G_B(K_{d,m,i})=\delta_{m-1}F_{d,m}\delta_m^{-1}=K_{d,m+1,0}$.
\end{proof}

In order to determine the gap sequences of all kernel words, we need introduce a new notation:
$\sigma_i(F_{d,\infty})$. For instance,
since $F_{3,\infty}=aaabaaabaaabaaaab\cdots$ and $\sigma_2(a,b):=(aab,a)$, then
$$\sigma_2(F_{3,\infty})=\underbrace{aab}_a\underbrace{aab}_a\underbrace{aab}_a \underbrace{a}_b\underbrace{aab}_a\underbrace{aab}_a\underbrace{aab}_a \underbrace{a}_b \underbrace{aab}_a\underbrace{aab}_a\underbrace{aab}_a \underbrace{a}_b \underbrace{aab}_a\underbrace{aab}_a\underbrace{aab}_a\underbrace{aab}_a \underbrace{a}_b\cdots$$



\begin{Theorem}[Gap sequence of kernel word $K_{d,m,i}$]\

(1) The kernel word $K_{d,m,i}$ has exactly two distinct gaps $G_A(K_{d,m,i})$ and $G_B(K_{d,m,i})$;

(2) The gap sequence $\{G_p(K_{d,m,i})\}_{p\geq1}$ is sequence $\mathcal{S}$ on $\{G_A(K_{d,m,i}),G_B(K_{d,m,i})\}$, where
\begin{equation*}
\mathcal{S}=\begin{cases}
\sigma_{d-i-1}(F_{d,\infty}),&0\leq i\leq d-2;\\
F_{d,\infty},&i=d-1.
\end{cases}
\end{equation*}
\end{Theorem}

\begin{proof} The proofs of $i=0$, $1\leq i\leq d-2$ and $i=d-1$ are similar, we take $1\leq i\leq d-2$ for example. The other two cases could be obtained by similar arguments.

\textbf{Step 1.} By the definition of sequence $\sigma_{d-i-1}(F_{d,\infty})$, we have
$\sigma_d(\alpha)=\alpha^d\beta$ and $\sigma_d(\beta)=\alpha$, where $\alpha=\sigma_{d-i-1}(a)=a^{d-i-1}b$ and $\beta=\sigma_{d-i-1}(b)=a$. That means
$$\sigma_d(a^{d-i-1}b)=[a^{d-i-1}b]^da,
~\sigma_d(a)=a^{d-i-1}b.$$
So we only need to prove the properties (1') and (2') below, where $\eta=\delta_{m}F_{d,m}^{d-1}F_{d,m-1}\sigma_d(\delta_{m})^{-1}$ is a shift. We write $G_A:=G_A(K_{d,m,i})$ and $G_B:=G_B(K_{d,m,i})$ for short.

(1') $\sigma_d([K_{d,m,i}G_A]^{d-i-1}[K_{d,m,i}G_B])=\eta\ast
\{[K_{d,m,i}G_A]^{d-i-1}[K_{d,m,i}G_B]\}^d[K_{d,m,i}G_A]
\ast\eta^{-1};$

(2') $\sigma_d(K_{d,m,i}G_A)=\eta\ast [K_{d,m,i}G_A]^{d-i-1}[K_{d,m,i}G_B]\ast\eta^{-1}$.

\vspace{0.2cm}

\textbf{Step 2.} When $1\leq i\leq d-2$, $G_A(K_{d,m,i})=K_{d,m,i-1}^{-1}$ and $G_B(K_{d,m,i})=K_{d,m+1,0}$. So

(a) $K_{d,m,i}G_A=\delta_{m}F_{d,m}\delta_{m}^{-1}$;

(b) $K_{d,m,i}G_B=\delta_{m}F_{d,m}^{i}F_{d,m-1}F_{d,m}\delta_{m}^{-1}$;

(c) $[K_{d,m,i}G_A(K_{d,m,i})]^{d-i-1}[K_{d,m,i}G_B(K_{d,m,i})]=\delta_{m}F_{d,m}^{d-1}F_{d,m-1}F_{d,m}\delta_{m}^{-1}$.

\vspace{0.2cm}

To prove property (1'), we have
\begin{equation*}
\begin{split}
&\sigma_d([K_{d,m,i}G_A]^{d-i-1}[K_{d,m,i}G_B])\\
=&\sigma_d(\delta_{m}F_{d,m}^{d-1}F_{d,m-1}F_{d,m}\delta_{m}^{-1})
=\sigma_d(\delta_{m})F_{d,m+1}^{d-1}F_{d,m}F_{d,m+1}\sigma_d(\delta_{m})^{-1}\\
&\{[K_{d,m,i}G_A]^{d-i-1}[K_{d,m,i}G_B]\}^d[K_{d,m,i}G_A]\\
=&[\delta_{m}F_{d,m}^{d-1}F_{d,m-1}F_{d,m}\delta_{m}^{-1}]^d\delta_{m}F_{d,m}\delta_{m}^{-1}
=\delta_{m}F_{d,m}^{d-1}F_{d,m-1}F_{d,m+1}^{d-1}F_{d,m}^2\delta_{m}^{-1}\\
=&\eta\ast \sigma_d(\delta_{m})
F_{d,m+1}^{d-1}F_{d,m}F_{d,m}\delta_{m}^{-1}\ast
\delta_{m}F_{d,m}^{d-1}F_{d,m-1}\sigma_d(\delta_{m})^{-1}\ast\eta^{-1}\\
=&\eta\ast \sigma_d(\delta_{m})
F_{d,m+1}^{d-1}F_{d,m}F_{d,m+1}\sigma_d(\delta_{m})^{-1}\ast\eta^{-1}
\end{split}
\end{equation*}

To prove property (2'), we have
\begin{equation*}
\begin{split}
&\sigma_d(K_{d,m,i}G_A)=\sigma_d(\delta_{m}F_{d,m}\delta_{m}^{-1})
=\sigma_d(\delta_{m})F_{d,m+1}\sigma_d(\delta_{m})^{-1}\\
&[K_{d,m,i}G_A]^{d-i-1}(K_{d,m,i}G_B)
=\delta_{m}F_{d,m}^{d-1}F_{d,m-1}F_{d,m}\delta_{m}^{-1}\\
=&\eta\ast \sigma_d(\delta_{m})F_{d,m}\delta_{m}^{-1}
\delta_{m}F_{d,m}^{d-1}F_{d,m-1}\sigma_d(\delta_{m})^{-1}
\ast \eta^{-1}
=\eta\ast \sigma_d(\delta_{m})F_{d,m+1}\sigma_d(\delta_{m})^{-1}
\ast \eta^{-1}
\end{split}
\end{equation*}

\textbf{Step 3.} We must prove the shift $\eta$ is allowable. In fact, $G_0(K_{d,m,i})=F_{d,m}\delta_m^{-1}$, then
$$\sigma_d(G_0(K_{d,m,i}))=F_{d,m+1}\sigma_d(\delta_m)^{-1}
=F_{d,m}\delta_m^{-1}\ast \delta_mF_{d,m}^{d-1}F_{d,m-1}\sigma_d(\delta_m)^{-1}
=G_0(K_{d,m,i})\ast \eta.$$
So under the substitution $\sigma_d$, there is a shift $\eta$.
\end{proof}

\noindent\textbf{Example.} Consider kernel word $K_{3,2,0}=aaaa$ in sequence $F_{3,\infty}$. $G_0(K_{3,2,0})=aaabaaabaaab$, two distinct gaps of the word $K_{3,2,0}$ are $$G_A(K_{3,2,0})=G_1(K_{3,2,0})=baaabaaab,~G_B(K_{3,2,0})=G_d(K_{3,2,0})=baaabaaabaaab.$$
\begin{equation*}
\begin{split}
F_{3,\infty}=&\underbrace{aaabaaabaaab}_{G_0}(aaaa)\underbrace{baaabaaab}_{A}(aaaa)\underbrace{baaabaaab}_A(aaaa)\underbrace{baaabaaabaaab}_B(aaaa)\\
&\underbrace{baaabaaab}_{A}(aaaa)\underbrace{baaabaaab}_A(aaaa)\underbrace{baaabaaabaaab}_B(aaaa)\underbrace{baaabaaab}_{A}(aaaa)\\
&\underbrace{baaabaaab}_A(aaaa)\underbrace{baaabaaabaaab}_B(aaaa)\underbrace{baaabaaab}_{A}(aaaa)\underbrace{baaabaaab}_{A}(aaaa)\\
&\underbrace{baaabaaab}_A(aaaa)\underbrace{baaabaaabaaab}_B(aaaa)\underbrace{baaabaaab}_{A}(aaaa)\underbrace{baaabaaab}_{A}(aaaa)\\
&\underbrace{baaabaaabaaab}_B(aaaa)\underbrace{baaabaaab}_{A}(aaaa)\underbrace{baaabaaab}_A(aaaa)\underbrace{baaabaaabaaab}_B(aaaa)\cdots\\
\end{split}
\end{equation*}
We see that the gap sequence is $\{G_p(aaaa)\}_{p\geq1}=AABAABAABAAABAABAAB\cdots$, which is the sequence $\sigma_{d-1}(F_{d,\infty})=\sigma_2(F_{3,\infty})$ on $\{A,B\}:=\{G_A(K_{d,m,0}),G_B(K_{d,m,0})\}$.


\vspace{0.5cm}

\stepcounter{section}

\noindent\textbf{\large{4.~Gaps and Gap Sequence of Envelope Word $E_{d,m,i}$}}

\vspace{0.4cm}

By similar arguments as in Section 3, we can determine the gaps and gap sequence of envelope word $E_{d,m,i}$. In this section, we only list the results.

\begin{Theorem} The prefix of $F_{d,\infty}$ before $E_{d,m,i,1}$ is $\varepsilon$, denoted by $G_0(E_{d,m,i})$.
\end{Theorem}

\begin{Theorem}[First two distinct gaps of envelope word $E_{d,m,i}$]\

Let $G_A(E_{d,m,i})=G_1(K_{d,m,i})$ and $B=\min\{p:G_p(E_{d,m,i})\neq G_1(E_{d,m,i})\}$,
then
\begin{equation*}
(1)~G_A(E_{d,m,i})=\begin{cases}
\varepsilon~(if~m=0)~\text{or}~E_{d,m-1,d-1}^{-1}~(if~m\geq1),&i=0;\\
E_{d,m,i-1}^{-1},&1\leq i\leq d-2;\\
b~(if~m=0)~\text{or}~E_{d,m-1,d-2}^{-1}~(if~m\geq1),&i=d-1.
\end{cases}
\end{equation*}
\begin{equation*}
(2)~G_B(E_{d,m,i})=\begin{cases}
b~(if~m=0)~\text{or}~E_{d,m-1,d-2}^{-1}~(if~m\geq1),&0\leq i\leq d-2;\\
E_{d,m,d-2}^{-1},&i=d-1.
\end{cases}
\end{equation*}
\begin{equation*}
(3)~B=\begin{cases}
d-i,&0\leq i\leq d-2;\\
d+1,&i=d-1.
\end{cases}
\end{equation*}
\end{Theorem}

\begin{Theorem}[Gap sequence of envelope word $E_{d,m,i}$]\

(1) The envelope word $E_{d,m,i}$ has exactly two distinct gaps $G_A(E_{d,m,i})$ and $G_B(E_{d,m,i})$;

(2) The gap sequence $\{G_p(E_{d,m,i})\}_{p\geq1}$ is sequence $\mathcal{S}$ on $\{G_A(E_{d,m,i}),G_B(E_{d,m,i})\}$, where
\begin{equation*}
\mathcal{S}=\begin{cases}
\sigma_{d-i-1}(F_{d,\infty}),&0\leq i\leq d-2;\\
F_{d,\infty},&i=d-1.
\end{cases}
\end{equation*}
\end{Theorem}


\vspace{0.5cm}

\stepcounter{section}

\noindent\textbf{\large{5.~The Relation Between $K_{d,m,i}$ and $E_{d,m,i}$}}

\vspace{0.4cm}

In Section 3 and 4, we determine the gaps and gap sequence of kernel word $K_{d,m,i}$ and envelope word $E_{d,m,i}$. Using these properties, we can give the relation between them.



\begin{proposition}[Position of $K_{d,m,i,p}$] The expression of position of the $(p+1)$-th occurrence of kernel word $K_{d,m,i}$, denote by $L(K_{d,m,i},p+1)$, is shown as below:
\begin{equation*}
\begin{cases}
(p+1)\ast f_{d,m}+|\sigma_{d-1}(F_{d,\infty})[1,p]|_b\ast f_{d,m-1},&i=0;\\
(p+1)\ast f_{d,m}+|\sigma_{d-i-1}(F_{d,\infty})[1,p]|_b\ast[i\ast f_{d,m}+f_{d,m-1}],&1\leq i\leq d-2;\\
(p+1)\ast f_{d,m}+|\sigma_{d+1}(F_{d,\infty})[1,p]|_a\ast[(d-1)\ast f_{d,m}+f_{d,m-1}],&i=d-1.
\end{cases}
\end{equation*}
\end{proposition}

\begin{proof} The proofs in the three cases are similar. Take $i=0$ for example.

(1) When $m=0$, by definition of kernel word, $K_{d,0,0}=a$. By Theorem 3.5, $$G_A(K_{d,0,0})=\varepsilon~\text{and}~G_B(K_{d,0,0})=K_{d,m+1,0}=b.$$
By Theorem 3.6, the gap sequence $\{G_p(K_{d,0,0})\}_{p\geq1}$ is sequence $\sigma_{d-1}(F_{d,\infty})$ on $\{G_A,G_B\}$. So
\vspace{-0.4cm}
\begin{equation*}
\begin{split}
&L(K_{d,0,0},p+1)=|G_0(K_{d,0,0})|+p\ast|K_{d,0,0}|+\sum_1^p|G_p(K_{d,0,0})|+1\\[-6pt]
=&|G_0(K_{d,0,0})|+p\ast|K_{d,0,0}|\\
&+|\sigma_{d-1}(F_{d,\infty})[1,p]|_a\ast|G_A(K_{d,0,0})|+|\sigma_{d-1}(F_{d,\infty})[1,p]|_b\ast|G_B(K_{d,0,0})|+1\\
=&f_{d,m}-1+p+|\sigma_{d-1}(F_{d,\infty})[1,p]|_a\ast0+|\sigma_{d-1}(F_{d,\infty})[1,p]|_b\ast1+1\\
=&f_{d,m}+p+|\sigma_{d-1}(F_{d,\infty})[1,p]|_b
\end{split}
\end{equation*}

(2) When $m\geq1$, by definition of kernel word, $K_{d,m,0}=\delta_mF_{d,m-1}\delta_{m-1}^{-1}$. By Theorem 3.5, $$G_A(K_{d,m,0})=K_{d,m-1,d-1}=\delta_{m-1}F_{d,m-1}^{d-1}F_{d,m-2}\delta_{m-2}^{-1},~G_B(K_{d,m,0})=K_{d,m+1,0}=\delta_{m+1}F_{d,m}\delta_{m}^{-1}.$$
By Theorem 3.6, the gap sequence $\{G_p(K_{d,m,0})\}_{p\geq1}$ is sequence $\sigma_{d-1}(F_{d,\infty})$ on $\{G_A,G_B\}$. So
\vspace{-0.4cm}
\begin{equation*}
\begin{split}
&L(K_{d,m,0},p+1)=|G_0(K_{d,m,0})|+p\ast|K_{d,m,0}|+\sum_1^p|G_p(K_{d,m,0})|+1\\[-6pt]
=&|G_0(K_{d,m,0})|+p\ast|K_{d,m,0}|\\
&+|\sigma_{d-1}(F_{d,\infty})[1,p]|_a\ast|G_A(K_{d,m,0})|+|\sigma_{d-1}(F_{d,\infty})[1,p]|_b\ast|G_B(K_{d,m,0})|+1\\
=&f_{d,m}-1+p\ast f_{d,m-1}\\
&+|\sigma_{d-1}(F_{d,\infty})[1,p]|_a\ast[(d-1)\ast f_{d,m-1}+f_{d,m-2}]
+|\sigma_{d-1}(F_{d,\infty})[1,p]|_b\ast f_{d,m}+1\\
=&(p+1)\ast f_{d,m}+|\sigma_{d-1}(F_{d,\infty})[1,p]_b\ast f_{d,m-1}
\end{split}
\end{equation*}

(3) When $m=0$, $f_{d,m}=f_{d,m-1}=1$, then
$$f_{d,m}+p+|\sigma_{d-1}(F_{d,\infty})[1,p]|_b=(p+1)\ast f_{d,m}+|\sigma_{d-1}(F_{d,\infty})[1,p]|_b\ast f_{d,m-1}.$$
So the two expressions are equivalent in this case.
\end{proof}

\begin{proposition}[position of $E_{d,m,i,p}$] The expression of position of the $(p+1)$-th occurrence of envelope word $E_{d,m,i}$, denoted by $L(E_{d,m,i},p+1)$, is shown as below:
\begin{equation*}
\begin{cases}
p\ast f_{d,m}+|\sigma_{d-1}(F_{d,\infty})[1,p]_b\ast f_{d,m-1}+1,&i=0;\\
p\ast f_{d,m}+|\sigma_{d-i-1}(F_{d,\infty})[1,p]|_b\ast[i\ast f_{d,m}+f_{d,m-1}]+1,&1\leq i\leq d-2;\\
p\ast f_{d,m}+|\sigma_{d+1}(F_{d,\infty})[1,p]|_a\ast[(d-1)\ast f_{d,m}+f_{d,m-1}]+1,&i=d-1.
\end{cases}
\end{equation*}
\end{proposition}

\begin{corollary} $L(K_{d,m,i},p+1)-L(E_{d,m,i},p+1)=f_{d,m}-1$. \end{corollary}

\begin{Theorem} The relation between $K_{d,m,i,p}$ and $E_{d,m,i,p}$ is
$$E_{d,m,i,p}=\mu_1\ast K_{d,m,i,p}\ast\mu_2,$$
where $\mu_1=\delta_{m+1}^{-1}K_{d,m+1,0}$ and $\mu_2=K_{d,m+1,0}\delta_{m+1}^{-1}$ are constant words depending only on $d$, $m$.
\end{Theorem}

\noindent\textbf{Remark.} Theorem 5.4 is stronger than Proposition 2.10.


\vspace{0.5cm}

\stepcounter{section}

\noindent\textbf{\large{6.~Gaps and Gap Sequence of Arbitrary Words}}

\vspace{0.4cm}

In this section, we give two versions of "uniqueness of kernel decomposition property": weak and strong, these properties make "kernel word" so important and so special. Using them, we can extend the properties about gaps and gap sequence from kernel and envelope word to arbitrary words in sequence $F_{d,m}$ for $d\geq2$.

\vspace{0.4cm}

\noindent\emph{6.1~Uniqueness of Kernel Decomposition Property}

\begin{lemma}\

(1) $K_{d,m,i}G_A(K_{d,m,i})K_{d,m,i}=K_{d,m,i+1}$ when $0\leq i\leq d-2$;

(2) $G_B(K_{d,m,i})=K_{d,m+1,0}$ when $0\leq i\leq d-2$;

(3) $G_A(K_{d,m,d-1})=K_{d,m+1,0}$ when $i=d-1$;

(4) $K_{d,m,d-1}G_B(K_{d,m,d-1})K_{d,m,d-1}=K_{d,m+2,0}$ when $i=d-1$.
\end{lemma}

\begin{proof} The properties (2) and (3) follow from Theorem 3.5 directly.
The proofs of (1) and (4) are similar, we prefer to take (1) for example.

When $1\leq i\leq d-2$, by the definition of kernel word, $K_{d,m,i}=\delta_m F_{d,m}^i F_{d,m-1}\delta_{m-1}^{-1}$; by Theorem 3.5, $G_A(K_{d,m,i})=K_{d,m,i-1}^{-1}$. So
\begin{equation*}
\begin{split}
&K_{d,m,i}G_A(K_{d,m,i})K_{d,m,i}\\
=&\delta_m F_{d,m}^i F_{d,m-1}\delta_{m-1}^{-1}\ast [\delta_m F_{d,m}^{i-1} F_{d,m-1}\delta_{m-1}^{-1}]^{-1}\ast \delta_m F_{d,m}^i F_{d,m-1}\delta_{m-1}^{-1}\\
=&\delta_m F_{d,m}F_{d,m}^i F_{d,m-1}\delta_{m-1}^{-1}=\delta_m F_{d,m}^{i+1} F_{d,m-1}\delta_{m-1}^{-1}=K_{d,m,i+1}
\end{split}
\end{equation*}

When $i=0$, $G_A(K_{d,m,i})=\varepsilon$ ($m=0$) and $K_{d,m-1,d-1}$ ($m\geq1$), the proof could be obtained by a similar argument as above.
\end{proof}

\begin{Theorem}[Uniqueness of kernel decomposition, weak]\

Let $Ker(\omega)=K_{d,m,i}$, then $\omega$ has an unique kernel decomposition as
$$\omega=\mu_1(\omega)\ast K_{d,m,i}\ast\mu_2(\omega),$$
where $\mu_1(\omega)$ and $\mu_2(\omega)$ are constant words depending only on $\omega$.
\end{Theorem}

\begin{proof} Let $Ker(\omega)=K_{d,m,i}$. Suppose $K_{d,m,i}$ occurs in $\omega$ twice. By Theorem 3.5, 3.6 and Lemma 6.1, we know:
(1) when $0\leq i\leq d-2$, $K_{d,m,i+1}$ or $K_{d,m+1,0}$ occurs in $\omega$;
(2) when $i=d-1$, $K_{d,m+1,0}$ or $K_{d,m+2,0}$ occurs in $\omega$.
All of the four factors are kernel word with higher order than $K_{d,m,i}$. It contradicts the hypotheses $Ker(\omega)=K_{d,m,i}$.
\end{proof}

\begin{proposition} Let $Ker(\omega)=K_{d,m,i}$, then $\omega\prec E_{d,m,i}$.
\end{proposition}

\begin{proof} By Theorem 5.4, for all $p\geq1$ we have $E_{d,m,i,p}=\mu_1\ast K_{d,m,i,p}\ast\mu_2$,
where $\mu_1=\delta_{m+1}^{-1}K_{d,m+1,0}$ and $\mu_2=K_{d,m+1,0}\delta_{m+1}^{-1}$. It give the structure around $K_{d,m,i,p}$.

Let $\omega=\mu'_1(\omega)\ast K_{d,m,i}\ast\mu'_2(\omega)$. Suppose $\omega\not\prec E_{d,m,i}$, then there are two cases:

Case 1: $\mu_1$ is a proper suffix of $\mu'_1(\omega)$;

Case 2: $\mu_2$ is a proper prefix of $\mu'_2(\omega)$.

Case 1 means $\delta_{m+1}\mu_1=K_{d,m+1,0}\prec\omega$ or $\delta_{m}\mu_1K_{d,m,i}\prec\omega$. Furthermore,
\begin{equation*}
\begin{split}
&\delta_{m}\mu_1K_{d,m,i}=\delta_{m}\delta_{m+1}^{-1}K_{d,m+1,0}K_{d,m,i}\\
=&\delta_{m}\delta_{m+1}^{-1}\ast\delta_{m+1}F_{d,m}\delta_m^{-1}\ast
\delta_mF_{d,m}^iF_{d,m-1}\delta_{m-1}^{-1}
=\delta_{m}F_{d,m}^{i+1}F_{d,m-1}\delta_{m-1}^{-1}\\
=&\begin{cases}
K_{d,m,i+1},&0\leq i\leq d-2;\\
\delta_{m}F_{d,m+1}\delta_{m-1}^{-1}=K_{d,m+1,0},&i=d-1.
\end{cases}
\end{split}
\end{equation*}
All of them mean there are kernel word with higher order than $K_{d,m,i}$ belong to $\omega$. It contradicts the hypotheses of $Ker(\omega)=K_{d,m,i}$.
The analysis of Case 2 is similar.
\end{proof}


\begin{Theorem}[Uniqueness of kernel decomposition, strong]\

Let $Ker(\omega)=K_{d,m,i}$, then $\omega_p$ has an unique kernel decomposition as
$$\omega_p=\mu_1(\omega)\ast K_{d,m,i,p}\ast\mu_2(\omega),$$
where $\mu_1(\omega)\triangleright\delta_{m+1}^{-1}K_{d,m+1,0}$ and $\mu_2(\omega)\triangleleft K_{d,m+1,0}\delta_{m+1}^{-1}$ are constant words depending only on $\omega$.
\end{Theorem}

\begin{proof} Since $Ker(\omega)=K_{d,m,i}$, by Proposition 6.3,
we have $K_{d,m,i}\prec\omega\prec E_{d,m,i}$. By Theorem 5.4, $E_{d,m,i,p}=\mu_1\ast K_{d,m,i,p}\ast\mu_2$, where $\mu_1=\delta_{m+1}^{-1}K_{d,m+1,0}$ and $\mu_2=K_{d,m+1,0}\delta_{m+1}^{-1}$.

So $Ker(\omega_p)=K_{d,m,i,p}$, the theorem holds.
\end{proof}

\noindent\textbf{Remark.} Theorem 6.2 (weak version) $\not\Rightarrow$ Theorem 6.4 (strong version).

For instance, let $\tau(a,b)=(abaa,abaa)$ and sequence $S=\tau^\infty(a)=abaaabaaabaa\cdots$. Let set $\mathcal{K}=\{a,b,aa,aaa\}\cup \{\omega\prec S:~|\omega|\geq5\}$. It is easy to check $\mathcal{K}$ is a kernel set of sequence $S$.
Consider factor $\omega=abaa$. $Ker(abaa)=aa$ and $abaa=ab\ast aa\ast\varepsilon$. So the weak version holds.
But $\omega_2=ab\ast aa_3\ast\varepsilon$, the strong version doesn't hold.

\begin{proposition}[Expression of $\omega$] Let $Ker(\omega)=K_{d,m,i}$, then $\omega$ has an expression as below:
$$\omega=F_{d,m}[x,f_{d,m}-1]\ast K_{d,m,i}\ast(\delta_{m+1}F_{d,m})[1,y],\eqno(\star)$$
where $1\leq x\leq f_{d,m}$ and $0\leq y\leq f_{d,m}-1$.
\end{proposition}

\begin{proof} Since $K_{d,m+1,0}=\delta_{m+1}F_{d,m}\delta_m^{-1}$, let $\omega[x,x-1]=\varepsilon$, then:

(1) all suffix of $\delta_{m+1}^{-1}K_{d,m+1,0}$ are $F_{d,m}[x,f_{d,m}-1]$ where $1\leq x\leq f_{d,m}$;

(2) all prefix of $K_{d,m+1,0}\delta_{m+1}^{-1}$ are $(\delta_{m+1}F_{d,m})[1,y]$ where $0\leq y\leq f_{d,m}-1$.
\end{proof}

\vspace{0.4cm}

\noindent\emph{6.2~Gaps and Gap Sequence of Arbitrary Words}

\begin{Theorem}[Gap sequence of $\omega$] Let $Ker(\omega)=K_{d,m,i}$, then

(1) Any factor $\omega$ has exactly two distinct gaps $G_A(\omega)$ and $G_B(\omega)$, where $A=1$ and
\begin{equation*}
B=\min\{p:G_p(\omega)\neq G_1(\omega)\}=\begin{cases}
d-i,&0\leq i\leq d-2;\\
d+1,&i=d-1.
\end{cases}
\end{equation*}

(2) The gap sequence $\{G_p(\omega)\}_{p\geq1}$ is sequence $\mathcal{S}$ on $\{G_A(\omega),G_B(\omega)\}$, where
\begin{equation*}
\mathcal{S}=\begin{cases}
\sigma_{d-i-1}(F_{d,\infty}),&0\leq i\leq d-2;\\
F_{d,\infty},&i=d-1.
\end{cases}
\end{equation*}
\end{Theorem}

\begin{proposition}[Relation between $G_p(\omega)$ and $G_p(Ker(\omega))$, $p\geq1$]\

Let $Ker(\omega)=K_{d,m,i}$, then
$$G_p(\omega)=\mu_2(\omega)^{-1}\ast G_p(K_{d,m,i})\ast\mu_1(\omega)^{-1},~p\geq1.$$
where $\mu_1(\omega)=F_{d,m}[x,f_{d,m}-1]$, $\mu_2(\omega)=(\delta_{m+1}F_{d,m})[1,y]$, $1\leq x\leq f_{d,m}$ and $0\leq y\leq f_{d,m}-1$.
\end{proposition}

\begin{proof} The proof of the Proposition will be easy by the following diagram.
\setlength{\unitlength}{1mm}
\begin{center}
\begin{picture}(120,15)
\linethickness{1pt}
\put(0,5){\line(1,0){15}}
\linethickness{3pt}
\put(15,5){\line(1,0){20}}
\linethickness{1pt}
\put(35,5){\line(1,0){55}}
\linethickness{3pt}
\put(90,5){\line(1,0){20}}
\linethickness{1pt}
\put(110,5){\line(1,0){10}}
\put(0,5){\line(0,1){10}}
\put(45,5){\line(0,1){10}}
\put(75,5){\line(0,1){10}}
\put(120,5){\line(0,1){10}}
\put(15,0){\line(0,1){10}}
\put(35,0){\line(0,1){10}}
\put(90,0){\line(0,1){10}}
\put(110,0){\line(0,1){10}}
\put(6,7){$\mu_1$}
\put(81,7){$\mu_1$}
\put(38,7){$\mu_2$}
\put(113,7){$\mu_2$}
\put(19,0){$K_{d,m,i,p}$}
\put(92,0){$K_{d,m,i,p+1}$}
\put(23,12){$\omega_p$}
\put(96,12){$\omega_{p+1}$}
\put(55,9){$G_p(\omega)$}
\put(51,0){$G_p(K_{d,m,i})$}
\put(21,13){\vector(-1,0){21}}
\put(28,13){\vector(1,0){17}}
\put(95,13){\vector(-1,0){20}}
\put(104,13){\vector(1,0){16}}
\put(54,10){\vector(-1,0){9}}
\put(67,10){\vector(1,0){8}}
\put(50,1){\vector(-1,0){15}}
\put(70,1){\vector(1,0){20}}
\end{picture}
\end{center}
\centerline{Fig. 6.1: The relation among $\omega_p$, $K_{d,m,i,p}$ $G_p(\omega)$ and $G_p(K_{d,m,i})$.}
\end{proof}

\begin{proposition}[Relation between $G_0(\omega)$ and $G_0(Ker(\omega))$]
Let $Ker(\omega)=K_{d,m,i}$, then
$$G_0(\omega)=G_0(K_{d,m,i})\ast\mu_1(\omega)^{-1},$$
where $\mu_1(\omega)=F_{d,m}[x,f_{d,m}-1]$ and $1\leq x\leq f_{d,m}$.
\end{proposition}

\begin{proof} The proof of the Proposition will be easy by the following diagram.
\setlength{\unitlength}{1mm}
\begin{center}
\begin{picture}(75,15)
\linethickness{1pt}
\put(0,5){\line(1,0){45}}
\linethickness{3pt}
\put(45,5){\line(1,0){20}}
\linethickness{1pt}
\put(65,5){\line(1,0){10}}
\put(30,5){\line(0,1){10}}
\put(75,5){\line(0,1){10}}
\put(0,0){\line(0,1){15}}
\put(45,0){\line(0,1){10}}
\put(65,0){\line(0,1){10}}
\put(36,7){$\mu_1$}
\put(68,7){$\mu_2$}
\put(49,0){$K_{d,m,i,1}$}
\put(53,12){$\omega_1$}
\put(10,9){$G_0(\omega)$}
\put(13,0){$G_0(K_{d,m,i})$}
\put(51,13){\vector(-1,0){21}}
\put(58,13){\vector(1,0){17}}
\put(9,10){\vector(-1,0){9}}
\put(22,10){\vector(1,0){8}}
\put(12,1){\vector(-1,0){12}}
\put(32,1){\vector(1,0){13}}
\end{picture}
\end{center}
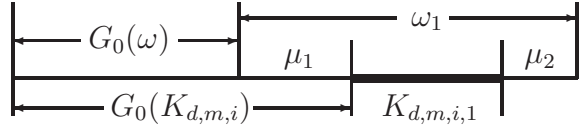
\centerline{Fig. 6.2: The relation among $\omega_1$, $K_{d,m,i,1}$, $G_0(\omega)$ and $G_0(K_{d,m,i})$.}
\end{proof}

\begin{corollary} Let $Ker(\omega)=K_{d,m,i}$ and $\omega$ has expression $(\star)$ in Proposition 6.5, then the prefix of $F_{d,\infty}$ before $\omega_1$ is
$F_{d,m}[1,x-1]$, denoted by $G_0(\omega)$.
\end{corollary}

Using expression $(\star)$ and the relation between $G_p(\omega)$
and $G_p(K_{d,m,i})$, we can give the expressions of $G_A(\omega)$ and $G_B(\omega)$. Since the proofs are simple and the expressions are complicated, we prefer to omit them.


\vspace{0.5cm}

\stepcounter{section}

\noindent\textbf{\large{7.~Combinatorial Properties of Factors}}

\vspace{0.4cm}

As applications, we will give some combinatorial properties of factors in sequence $F_{d,\infty}$.

\vspace{0.4cm}

\noindent\emph{7.1 Palindrome}

\vspace{0.4cm}

Since both kernel words $K_{d,m,i}$ and envelope words $E_{d,m,i}$ are palindromes, we can determine all palindromes with kernel $K_{d,m,i}$.

\begin{property}[Palindrome] Let $\omega$ have expression $(\star)$ in Proposition 6.5, then
$$\omega\text{ is palindrome }\Leftrightarrow x+y=f_{d,m}.$$
\end{property}

\begin{proof} $\omega$ is palindrome $\Leftrightarrow$ $|F_{d,m}[x,f_{d,m}-1]|=|(\delta_{m+1}F_{d,m})[1,y]|$, i.e. $f_{d,m}-x=y$.\end{proof}

\vspace{0.4cm}

\noindent\emph{7.2 Whether $\omega\prec F_{d,\infty}$ or Not ?}

\vspace{0.4cm}

By Property 6.3, we know if $\omega\prec F_{d,\infty}$, $Ker(\omega)=K_{d,m,i}$, then $\omega\prec E_{d,m,i}$. On the other hand, since $E_{d,m,i}\prec F_{d,\infty}$, then $\omega\prec E_{d,m,i}\Rightarrow\omega\prec F_{d,\infty}$. So we can get the property below.

\begin{property} Let $Ker(\omega)=K_{d,m,i}$, then
$\omega\prec F_{d,\infty}\Leftrightarrow \omega\prec E_{d,m,i}.$
\end{property}

\noindent\textbf{Remark.} 1. When $\omega\not\prec F_{d,\infty}$, the kernel word maybe occurs in $\omega$ more than once. For instance, in sequence $F_{2,\infty}$, let $\omega=aaaa$, then $Ker(\omega)=aaa$, but $\omega=\underline{aaa}a=a\underline{aaa}$.

2. $\omega$ has an unique kernel decomposition $\not\Rightarrow \omega\prec F_{d,\infty}$. For instance, in sequence $F_{2,\infty}$, let $\omega=aaabb$,
then $Ker(\omega)=aaa$ and $\omega=\underline{aaa}bb$, but $\omega\not\prec F_{d,\infty}$.

\vspace{0.4cm}

\noindent\emph{7.3 Power, Overlap and Separate Properties Between $\omega_p$ and $\omega_{p+1}$}

\vspace{0.4cm}

Notice that if we consider the position of $\omega_p\in F_{d,\infty}$, the factors $\omega_p$ and $\omega_q~(p\neq q)$ are distinct.
In fact, $\omega_p$ should be regarded as two variables $\omega$ and $p$, where $\omega$ is the factor and $p$ indicates the position of $\omega$.

Let $Ker(\omega)=K_{d,m,i}$, $0\leq i\leq d-1$. When $m=0$, $\omega\in\{a^i,i=1,2,\ldots,d\}$.
When $m\geq1$, by Theorem 3.5 and 4.2, all factor $\omega\in F_{d,\infty}$ can divide into several types according to the different lengths of gaps. We denote those types by $T_{\alpha,\beta}$.

%
%

\begin{definition}[Types] The sets $T_{\alpha,\beta}$ are defined as follow:

$T_{0,1}=\{\omega\in F_{d,\infty}: Ker(\omega)=K_{d,m,0}, G_A(\omega)>0, G_B(\omega)>0\}$;

$T_{0,2}=\{\omega\in F_{d,\infty}: Ker(\omega)=K_{d,m,0}, G_A(\omega)=0, G_B(\omega)>0\}$;

$T_{0,3}=\{\omega\in F_{d,\infty}: Ker(\omega)=K_{d,m,0}, G_A(\omega)<0, G_B(\omega)>0\}$;

$T_{0,4}=\{\omega\in F_{d,\infty}: Ker(\omega)=K_{d,m,0}, G_A(\omega)<0, G_B(\omega)=0\}$;

$T_{0,5}=\{\omega\in F_{d,\infty}: Ker(\omega)=K_{d,m,0}, G_A(\omega)<0, G_B(\omega)<0\}$;

$T_{i,1}=\{\omega\in F_{d,\infty}: Ker(\omega)=K_{d,m,i}, G_A(\omega)<0, G_B(\omega)>0\}$, $1\leq i\leq d-2$;

$T_{i,2}=\{\omega\in F_{d,\infty}: Ker(\omega)=K_{d,m,i}, G_A(\omega)<0, G_B(\omega)=0\}$, $1\leq i\leq d-2$;

$T_{i,3}=\{\omega\in F_{d,\infty}: Ker(\omega)=K_{d,m,i}, G_A(\omega)<0, G_B(\omega)<0\}$, $1\leq i\leq d-2$;

$T_{d-1,1}=\{\omega\in F_{d,\infty}: Ker(\omega)=K_{d,m,d-1}, G_A(\omega)>0, G_B(\omega)<0\}$;

$T_{d-1,2}=\{\omega\in F_{d,\infty}: Ker(\omega)=K_{d,m,d-1}, G_A(\omega)=0, G_B(\omega)<0\}$;

$T_{d-1,3}=\{\omega\in F_{d,\infty}: Ker(\omega)=K_{d,m,d-1}, G_A(\omega)<0, G_B(\omega)<0\}$.
\end{definition}

\begin{property}[Types] Let $m\geq1$, $1\leq x\leq f_{d,m}$ and $0\leq y\leq f_{d,m}-1$, then

(1) When $Ker(\omega)=K_{d,m,0}$,

$\omega\in T_{0,1}\Leftrightarrow G_A(\omega)>0,G_B(\omega)>0 \Leftrightarrow f_{d,m-1}\leq x-y\leq f_{d,m}$;

$\omega\in T_{0,2}\Leftrightarrow G_A(\omega)=0,G_B(\omega)>0 \Leftrightarrow x-y=f_{d,m-1}$;

$\omega\in T_{0,3}\Leftrightarrow G_A(\omega)<0,G_B(\omega)>0 \Leftrightarrow 0<x-y<f_{d,m-1}$;

$\omega\in T_{0,4}\Leftrightarrow G_A(\omega)<0,G_B(\omega)=0 \Leftrightarrow x-y=0$;

$\omega\in T_{0,5}\Leftrightarrow G_A(\omega)<0,G_B(\omega)<0 \Leftrightarrow 2-f_{d,m}\leq x-y<0$.

(2) When $Ker(\omega)=K_{d,m,i}$, $1\leq i\leq d-2$,

$\omega\in T_{i,1}\Leftrightarrow G_A(\omega)<0,G_B(\omega)>0 \Leftrightarrow 0\leq x-y\leq f_{d,m}$;

$\omega\in T_{i,2}\Leftrightarrow G_A(\omega)<0,G_B(\omega)=0 \Leftrightarrow x-y=0$;

$\omega\in T_{i,3}\Leftrightarrow G_A(\omega)<0,G_B(\omega)<0 \Leftrightarrow 2-f_{d,m}\leq x-y<0$.

(3) When $Ker(\omega)=K_{d,m,d-1}$,

$\omega\in T_{d-1,1}\Leftrightarrow G_A(\omega)>0,G_B(\omega)<0 \Leftrightarrow 0\leq x-y\leq f_{d,m}$;

$\omega\in T_{d-1,2}\Leftrightarrow G_A(\omega)=0,G_B(\omega)<0 \Leftrightarrow x-y=0$;

$\omega\in T_{d-1,3}\Leftrightarrow G_A(\omega)<0,G_B(\omega)<0 \Leftrightarrow 2-f_{d,m}\leq x-y<0$.
\end{property}

\begin{proof} Let $Ker(\omega)=K_{d,m,i}$, then

(a) $|G_A(\omega)|=|G_A(K_{d,m,i})|-|\mu_1|-|\mu_2|=|G_A(K_{d,m,i})|-f_{d,m}+x-y$;

(b) $|G_B(\omega)|=|G_B(K_{d,m,i})|-|\mu_1|-|\mu_2|=|G_B(K_{d,m,i})|-f_{d,m}+x-y$.

When $i=0$, consider $G_A(K_{d,m,i})=K_{d,m-1,d-1}$ and $G_B(K_{d,m,i})=K_{d,m+1,0}$, we have:

(1) $|G_A(\omega)|=f_{d,m-1}^{d-1}f_{d,m-2}-f_{d,m}+x-y$, then
$G_A(\omega)\geq0\Leftrightarrow x-y\geq f_{d,m-1}$;

(2) $|G_B(\omega)|=f_{d,m}-f_{d,m}+x-y$, then
$G_A(\omega)\geq0\Leftrightarrow x-y\geq0$.

When $1\leq i\leq d-2$, consider $G_A(K_{d,m,i})=K_{d,m,i-1}^{-1}$ and $G_B(K_{d,m,i})=K_{d,m+1,0}$, we have:

(3) $|G_A(\omega)|\leq|G_A(K_{d,m,i})|<0$ for all $x$ and $y$;

(4) $|G_B(\omega)|=f_{d,m}-f_{d,m}+x-y$, then
$G_A(\omega)\geq0\Leftrightarrow x-y\geq0$.

When $i=d-1$, consider $G_A(K_{d,m,i})=K_{d,m+1,0}$ and $G_B(K_{d,m,i})=K_{d,m,d-2}^{-1}$, we have:

(5) $|G_A(\omega)|=f_{d,m}-f_{d,m}+x-y$, then
$G_B(\omega)\geq0\Leftrightarrow x-y\geq0$.

(6) $|G_B(\omega)|\leq|G_B(K_{d,m,i})|<0$ for all $x$ and $y$.
\end{proof}

Obviously, the disjoint union of sets $T_{\alpha,\beta}$ consists of all factors with kernel $K_{d,m,i}$ for $m\geq1$.

\begin{definition} For $\omega\prec F_{d,\infty}$ and $p\geq1$, we can define six sets as below:

$\mathcal{P}_1:=\{(\omega,p):\omega_p~and~\omega_{p+1}~are~adjacent\}$,
$\mathcal{P}_2:=\{\omega:\exists ~p~s.t.~(\omega,p)\in\mathcal{P}_1\}$;

$\mathcal{S}_1:=\{(\omega,p):\omega_p~and~\omega_{p+1}~are~separated\}$;
$\mathcal{S}_2:=\{\omega:\exists ~p~s.t.~(\omega,p)\in\mathcal{S}_1\}$;

$\mathcal{O}_1:=\{(\omega,p):\omega_p~and~\omega_{p+1}~are~overlapped\}$;
$\mathcal{O}_2:=\{\omega:\exists ~p~s.t.~(\omega,p)\in\mathcal{O}_1\}$.
\end{definition}

\begin{definition}[$\Gamma_{i,\gamma}$] Let $\gamma=a$ or $b$, then
\begin{equation*}
\Gamma_{i,\gamma}=\begin{cases}
\{p\in\mathbb{N}:~\sigma_{d-i-1}(F_{d,\infty})[p]=\gamma\},&0\leq i\leq d-2;\\
\{p\in\mathbb{N}:~F_{d,\infty}[p]=\gamma\},&i=d-1.
\end{cases}
\end{equation*}

\end{definition}

\begin{property}[Combinatorial properties, strong] For $m\geq1$,

(1) $\mathcal{P}_1=(T_{0,2},\Gamma_{0,a})\cup (T_{0,4},\Gamma_{0,b})\cup[\cup_{i=1}^{d-2}(T_{i,2},\Gamma_{i,b})]\cup(T_{d-1,2},\Gamma_{d-1,a})$;

(2) $\mathcal{S}_1=(T_{0,1},\mathbb{N})\cup (T_{0,2}\cup T_{0,3},\Gamma_{0,b})\cup[\cup_{i=1}^{d-2}(T_{i,1},\Gamma_{i,b})]\cup(T_{d-1,1},\Gamma_{d-1,a})$;

(3) $\mathcal{O}_1=(T_{0,3}\cup T_{0,4},\Gamma_{0,a})\cup [\cup_{i=1}^{d-2}(T_{i,1}\cup T_{i,2},\Gamma_{i,a})]\cup (\cup_{i=1}^{d-2}T_{i,3}\cup T_{0,5}\cup T_{d-1,3},\mathbb{N})\cup(T_{d-1,1}\cup T_{d-1,2},\Gamma_{d-1,b})$.
\end{property}

\begin{property}[Combinatorial properties, weak] Let $ker(\omega)=K_{d,m,i}$, $m\geq1$, and let $\omega$ have expression $(\star)$ in Proposition 6.5, then

(1) $\mathcal{P}_2=\{i=0,~x-y=f_{d,m}-1\}\cup\{x-y=0\}$;

(2) $\mathcal{S}_2=\{0\leq x-y\leq f_{d,m}\}$;

(3) $\mathcal{O}_2=\{i=0,~2-f_{d,m}\leq x-y<f_{d,m-1}\}\cup\{1\leq i\leq d-1\}$.
\end{property}


\vspace{0.5cm}

\noindent\textbf{\large{Acknowledgments}}

\vspace{0.4cm}

The research is supported by the Grant NSFC No.11271223, No.11371210 and No.11326074.

\end{CJK*}

\begin{thebibliography}{AB}

\bibitem{AB2005} I.M.Ara$\acute{u}$jo, V.Bruy$\grave{e}$re, Words derivated from Sturmian words, Theor. Comput. Sci. 340 (2005) 204-219.

\bibitem{AS2003}  J.M.Allouche, J.Shallit, Automatic sequences: Theory, applications, generalizations. Cambridge University Press, Cambridge, 2003.

\bibitem{B1966} J.Berstel, Recent results in Sturmian words, in J.Dassow, A.Salomaa (Eds.), Developments in Language Theory, World Scientific, Singapore, 1966, pp.13-24.

\bibitem{B1980} J.Berstel, Mot de Fibonacci, S$\acute{e}$minaire d'informatique th$\acute{e}$rique, L.I.T.P., Paris, Ann$\acute{e}$e 1980/1981, pp.57-78.

\bibitem{BPS2008} L.Balkov$\acute{a}$, E.Pelantov$\acute{a}$, W.Steiner, Sequences with constant number of return words, Monatsh Math. 155 (2008) 251-263.

\bibitem{CW2003} W.-T.Cao, Z.-Y.Wen, Some properties of the factors of Sturmian sequences,
Theor. Comput. Sci. 304 (2003) 365-385.

\bibitem{CH2010} W.-F.Chuan, H.-L.Ho, Factors of characteristic words: Location and decompositions,
Theor. Comput. Sci. 411 (2010) 31-33.

\bibitem{D1998} F.Durand, A characterization of substitutive sequences using return words, Discrete Math. 179 (1998) 89-101.

\bibitem{HW2014} Y.-K.Huang, Z.-Y.Wen, Gap Sequence of Factors of Fibonacci Sequence, arXiv:1404.4269.

\bibitem{IY1990} S. Ito, S. Yasutomi, On continued fractions, substitutions and characteristic sequence, Japan. J. Math., 16 (1990), pp. 287-306.


\bibitem{L1983} M.Lothaire, Combinatorics on words, in: Encyclopedia of Mathematics and its applications, Vol.17, Addison-Wesley, Reading, MA, 1983.

\bibitem{L2002} M.Lothaire, Algebraic combinatorics on words, Cambridge Univ. Press, Cambridge, 2002.

\bibitem{M1991} F. Mignosi, On the number of factors of Sturmian words, Theoret . Comput. Sci., 88 (1991), pp. 71-84.

\bibitem{WW1993} Z.-X.Wen, Z.-Y.Wen, Some studies of factors of infinite words generated by invertible substitution, in: A. Barlotti, M.Delest, R. Pinzani (Eds.), Proc. Fifth Conf. Formal Power Seres and Algebraic Combinatorics, 1993, pp. 455-466.

\bibitem{WW1994} Z.-X.Wen, Z.-Y.Wen, Some properties of the singular words of the Fibonacci word, European J. Combin. 15 (1994) 587-598.

\bibitem{V2001} L.Vuillon, A characterization of Sturmian words by return words, European J. Combin. 22 (2001) 263-275.

\end{thebibliography}
\end{document}